\documentclass[letterpaper]{amsart} 

\usepackage{amsmath}
\usepackage[all]{xy}

\usepackage{psfrag}
\usepackage{epsfig}
 
\numberwithin{equation}{section}

\newtheorem{theorem}{Theorem}[section] 
\newtheorem{proposition}[theorem]{Proposition}

\newtheorem{lemma}[theorem]{Lemma} 
 
\theoremstyle{definition}

\newtheorem{definition}[theorem]{Definition}

\def\cc{\mathbf{c}}

\def\TT{\mathbb{T}}

\def\ZZ{\mathbb{Z}}

\newcommand{\overunder}[2]{
\!\begin{array}{c}
\scriptstyle{#1}\\[-.1in]
-\!\!\!-\!\!\!-\\[-.1in]
\scriptstyle{#2}
\end{array}
\!
}
\begin{document}

\title{Maximal green sequences of skew-symmetrizable $3\times 3$ matrices}

\author{Ahmet I. Seven}

\address{Middle East Technical University, Department of Mathematics, 06800, Ankara, Turkey}
\email{aseven@metu.edu.tr}

\thanks{The author's research was supported in part by the Turkish Research Council (TUBITAK)}


\date{July 26, 2012}



\begin{abstract}
Maximal green sequences are particular sequences of mutations of skew-symmetrizable matrices which were introduced by Keller in the context of quantum dilogarithm identities and independently by Cecotti-C\'ordova-Vafa in the context of supersymmetric gauge theory. In this paper we study maximal green sequences of skew-symmetrizable $3\times 3$ matrices. We show that such a matrix with a mutation-cyclic diagram does not have any maximal green sequences. We also obtain some properties of maximal green sequences of skew-symmetrizable matrices with mutation-acyclic diagrams.


\end{abstract}
\maketitle

\section{Introduction}
\label{sec:intro}

Maximal green sequences are particular sequences of mutations of skew-symmetrizable matrices. They were used in \cite{Ke2} to obtain quantum dilogarithm identities. Moreover, the same sequences appeared in theoretical physics where they yield the complete spectrum of a BPS particle, see \cite[Section~4.2]{CCV}. In this paper we study the maximal green sequences of skew-symmetrizable $3\times 3$ matrices. 
We show that 
those matrices with a mutation-cyclic diagram do not have any maximal green sequences. We also obtain some properties of maximal green sequences of skew-symmetrizable matrices with mutation-acyclic diagrams.

To be more specific, we need some terminology. Let us recall that a skew-symmetrizable matrix $B$ is an $n\times n$ integer matrix such that $DB$ is skew-symmetric for some diagonal matrix $D$ with positive diagonal entries. We consider pairs $(\cc, B)$, where $B$ is a skew-symmetrizable integer matrix and $\cc=(\cc_1,...,\cc_n)$ such that each $\cc_i=(c_1,...,c_n) \in \ZZ^n$ is non-zero. Motivated by the structural theory of cluster algebras, we call such a pair $(\cc, B)$ a $Y$-seed. Then, for $k = 1, \dots, n$, the \emph{$Y$-seed mutation} $\mu_k$ transforms
$(\cc, B)$ into the $Y$-seed $\mu_k(\cc, B)=(\cc', B')$ defined as follows \cite[Equation~(5.9)]{CAIV}, where we use the notation $[b]_+ = \max(b,0)$:
\begin{itemize}
\item
The entries of the exchange matrix $B'=(B'_{ij})$ are given by
\begin{equation}
\label{eq:matrix-mutation}
B'_{ij} =
\begin{cases}
-B_{ij} & \text{if $i=k$ or $j=k$;} \\[.05in]
B_{ij} + [B_{ik}]_+ [B_{kj}]_+ - [-B_{ik}]_+ [-B_{kj}]_+
 & \text{otherwise.}
\end{cases}
\end{equation}
\item
The tuple $\cc'=(\cc_1',\dots,\cc_n')$ is given by
\begin{equation}
\label{eq:y-mutation}
\cc'_i =
\begin{cases}
-\cc_{i} & \text{if $i = k$};\\[.05in]
\cc_i+[sgn(\cc_k)B_{k,i}]_+\cc_k
 & \text{if $i \neq k$}.
\end{cases}
\end{equation}
\end{itemize}
This transformation is involutive; furthermore, $B'$ is skew-symmetrizable with the same choice of $D$. We also use the notation $B' = \mu_k(B)$ (in \eqref{eq:matrix-mutation}) and call the transformation $B \mapsto B'$ the \emph{matrix mutation}. This operation is involutive, so it defines a \emph{mutation-equivalence} relation on skew-symmetrizable matrices.

We use the $Y$-seeds in association with the vertices of a regular tree.  To be more precise, let $\TT_n$ be an \emph{$n$-regular tree} whose edges are labeled by the numbers $1, \dots, n$, so that the $n$ edges emanating from each vertex receive different labels. We write $t \overunder{k}{} t'$ to indicate that vertices $t,t'\in\TT_n$ are joined by an edge labeled by~$k$.
Let us fix an initial seed at a vertex $t_0$ in $\TT^n$ and assign the (initial) $Y$-seed $(\cc_0,B_0)$, where $\cc_0$ is the tuple of standard basis. 
This defines a \emph{$Y$-seed pattern} on $\TT_n$, i.e. an assignment of a seed $(\cc_t, B_t)$ to every vertex $t \in \TT_n$, such that the seeds assigned to the endpoints of any edge $t \overunder{k}{} t'$ are obtained from each other by the seed mutation~$\mu_k$.
We write:
\begin{equation}
\label{eq:seed-labeling}
\cc_t =\cc= (\cc_{1}\,,\dots,\cc_{n})\,,\quad
B_t=B = (B_{ij})\,.
\end{equation}
We refer to $B$ as the \emph{exchange matrix} and $\cc$ as the \emph{$\cc$-vector} tuple of the $Y$-seed. It is conjectured that $\cc$-vectors have the following \emph{sign coherence property}:
\begin{equation}
\label{eq:C-sign-coherence}
\text{each vector $\cc_{j}$ has either all entries nonnegative or all entries nonpositive.}
\end{equation}
This conjectural property \eqref{eq:C-sign-coherence} has been proved in \cite{DWZ2} for the case of \emph{skew-symmetric} exchange matrices, using
\emph{quivers with potentials} and their representations.


We need a bit more terminology. The \emph{diagram} of a skew-symmetrizable $n\times n$ matrix ${B}$ is the directed graph $\Gamma ({B})$ defined as follows: the vertices of $\Gamma ({B})$ are the indices $1,2,...,n$ such that there is a directed edge from $i$ to $j$ if and only if ${B}_{j,i} > 0$, and this edge is assigned the weight $|B_{ij}B_{ji}|\,$. By a {subdiagram} of $\Gamma(B)$, we always mean a diagram obtained from $\Gamma(B)$ by taking an induced (full) directed subgraph on a subset of vertices and keeping all its edge weights the same as in $\Gamma(B)$. By a cycle in $\Gamma(B)$ we mean a subdiagram whose vertices can be labeled by elements of $\ZZ/m\ZZ$ so that the edges betweeen them are precisely $\{i,i+1\}$ for $i \in  \ZZ/m\ZZ$. Let us also note that if $B$ is skew-symmetric then it is also represented, alternatively, by a quiver whose vertices are the indices $1,2,...,n$ and there are $B_{j,i}>0$ many arrows from $i$ to $j$. This quiver uniquely determines the corresponding skew-symmetric matrix, so mutation of skew-symmetric matrices can be viewed as a "quiver mutation".
We call a diagram $\Gamma$ \emph{mutation-acyclic} if it is mutation-equivalent to an acyclic diagram (i.e. a diagram which has no oriented cycles at all); otherwise we call it \emph{mutation-cyclic}. 



Now we can recall the notion of a green sequence \cite{Ke2}:
\begin{definition}\label{def:green}
Let $B_0$ be a skew-symmetrizable $n\times n$  matrix. A \emph{green sequence for $B_0$} is a sequence $\mathbf{i} = (i_1, \ldots, i_l)$ such that, 
for any $1 \leq k \leq l$ with $(\cc,B)=\mu_{i_{k-1}} \circ \cdots \circ \mu_{i_1}(\cc_0,B_0)$, we have $\cc_{i_{k}}>0$ ; here if $k=1$, then we take $(\cc,B)=(\cc_0,B_0)$.

A green sequence $\mathbf{i} = (i_1, \ldots, i_l)$ is maximal if, for $(\cc,B)=\mu_{i_{l}} \circ \cdots \circ \mu_{i_1}(\cc_0,B_0)$, we have $\cc_k<0$ for all $k=1,...,n$.

\end{definition}

In this paper, we study the maximal green sequences in the basic case of size $3$ skew-symmetrizable matrices. Our first result is the following:

\begin{theorem}\label{th:cyclic}
Suppose that $B$ is a skew-symmetrizable $3\times 3$ matrix. Under the assumption \eqref{eq:C-sign-coherence},
if $\Gamma(B)$ is mutation-cyclic, then $B$ does not have any maximal green sequences.
\end{theorem}
\noindent

For skew-symmetrizable matrices with mutation-acyclic diagrams, we have the following result:

\begin{theorem}\label{th:acyclic}
Suppose that $B$ is a skew-symmetrizable $3\times 3$ matrix.
Suppose also that $\Gamma(B)$ is mutation-acyclic and $\mathbf{i} = (i_1, \ldots, i_l)$ is a maximal green sequence for $B$. Let $B=B_0$ and, for $j=1,...,l$, let $B_j=\mu_{{i_j}} \circ \cdots \circ \mu_{i_1}(B)$. Then, under the assumption \eqref{eq:C-sign-coherence}, the diagram $\Gamma(B_j)$ is acyclic for some $0\leq j\leq l$ .

\end{theorem}

We also have the following result, which writes the initial exchange matrix in terms of a $Y$-seed:

\begin{theorem}\label{th:cylic+acyclic}
For any skew-symmetrizable $3\times 3$ matrix $B$, let $u=u(B)=(a,b,c)$ be defined as follows: if $\Gamma(B)$ is cyclic, then $(a,b,c)=(d_2|B_{2,3}|,d_3|B_{3,1}|,d_1|B_{1,2}|)$; if $\Gamma(B)$ is acyclic, then $(a,b,c)=(\pm d_2|B_{2,3}|,\pm d_3|B_{3,1}|,\pm d_1|B_{1,2}|)$ such that the coordinates corresponding to the source and sink have the same sign and the remaining coordinate has the opposite sign. 

Suppose now that $(\cc',B')$ is a $Y$-seed with respect to the initial seed $(\cc_0,B)$ and let $u'=u(B')=(a',b',c')$. 
Then, under the assumption \eqref{eq:C-sign-coherence}, we have $a'\cc'_1+b'\cc'_2+c'\cc'_3=\pm (a,b,c)$. Furthermore, if $\Gamma(B)$ is mutation-cyclic then $a'\cc'_1+b'\cc'_2+c'\cc'_3=(a,b,c)$.

\end{theorem}

\section{Proofs of main results}
\label{sec:proof}
The matrices in this section are skew-symmetrizable $3\times 3$ matrices. We also assume that \eqref{eq:C-sign-coherence} is satisfied. First we note the following two properties, which can be easily checked using the definitions:

\begin{proposition} 
\label{prop:coord_change} 
Suppose that $(\cc,B)$ is a $Y$-seed (with respect to an initial $Y$-seed). Suppose also that the coordinate vector of $u$ with respect to $\cc$ is $(a_1,...,a_n)$. Let $(\cc', B')=\mu_k(\cc, B)$ and $(a'_1,...,a'_n)$ be the coordinates of $u$ with respect to $\cc'$. Then $a_i=a'_i$ if $i \ne k$ and $a'_k=-a_k+\sum a_i[sgn(\cc_k)B_{k,i}]_+$, where the sum is over all $i \ne k$.


\end{proposition}

\begin{proposition} 
\label{prop:+-} 





Suppose that $B$ is a skew-symmetrizable $3\times 3$ matrix $B$. Let $u=u(B)=(a,b,c)$ be defined as follows: if $\Gamma(B)$ is cyclic, then $(a,b,c)=(d_2|B_{2,3}|,d_3|B_{3,1}|,d_1|B_{1,2}|)$; if $\Gamma(B)$ is acyclic, then $(a,b,c)=(\pm d_2|B_{2,3}|,\pm d_3|B_{3,1}|,\pm d_1|B_{1,2}|)$ such that the coordinates corresponding to the source and sink have the same sign and the remaining coordinate has the opposite sign. 
Then the vector $u$ is a radical vector for $B$, i.e. $Bu=0$. 


\end{proposition}

We also need the following two lemmas to prove our results:

\begin{lemma} 
\label{lem:++}

Suppose that $(\cc,B)$ is a $Y$-seed (with respect to an initial $Y$-seed) such that $\Gamma(B)$ is \emph{cyclic} and let $u$ be the vector whose coordinate vector with respect to the basis $\cc$ is $(d_2|B_{2,3}|,d_3|B_{3,1}|,d_1|B_{1,2}|)$. Let $(\cc',B')=\mu_{{k}} (\cc,B)$. Then we have the following:

If the diagram $\Gamma(B')$ is also cyclic, then the coordinate vector of $u$ with respect to the basis $\cc'$ is $(d_2|B'_{2,3}|,d_3|B'_{3,1}|,d_1|B'_{1,2}|)$.


If the diagram $\Gamma(B')$ is acyclic, then the coordinate vector of $u$ with respect to the basis $\cc'$ is obtained from  $(d_2|B'_{2,3}|,d_3|B'_{3,1}|,d_1|B'_{1,2}|)$ by multiplying the $k$-th coordinate by $-1$.
(Note that the vertex $k$ is neither a source nor a sink in $\Gamma(B')$.)

\end{lemma} 

\begin{proof}
Assume without loss of generality that $sgn(\cc_k)=sgn(B_{k,i})$.
Then $\cc'_i=\cc_i+|B_{k,i}|\cc_k$, $\cc'_j=\cc_j$, $\cc'_k=-\cc_k$. Thus the k-th coordinate of $u$ with respect to $\cc'$ will be 
$-d_i|B_{i,j}|+|B_{k,i}|d_k|B_{k,j}|=-d_i|B_{i,j}|+|B_{i,k}|d_i|B_{k,j}|=d_i(-|B_{i,j}|+|B_{i,k}||B_{k,j}|)$ (Proposition~\ref{prop:coord_change} 
) , 
the other coordinates are the same (note that the $k$-th coordinate of $u$ with respect to $\cc$ is $d_i|B_{i,j}|=d_j|B_{j,i}|$). 
Thus, to prove the statement in the first part, it is enough to show $|B'_{i,j}|=-|B_{i,j}|+|B_{i,k}||B_{k,j}|$ (Recall that the skew-symmetrizing matrix $D$ is preserved under mutations). For convenience, we investigate in cases:

\noindent
Case 1. $B_{k,i}>0$. Then $B_{i,j}>0$ and $B_{j,k}>0$ (so $B_{i,k}<0,B_{j,i}<0,B_{k,j}<0$).
Then, since $\Gamma(B')$ is cyclic, we have $B'_{k,i}<0,B'_{i,j}<0,B'_{j,k}<0$.
By \eqref{eq:matrix-mutation}, we have
$$B'_{ij} =B_{ij}+[B_{ik}]_+ [B_{kj}]_+ - [-B_{ik}]_+ [-B_{kj}]_+=B_{ij}-[-B_{ik}][-B_{kj}]=B_{ij}-B_{ik}B_{kj}.$$
Since $B'_{i,j}<0$ by assumption in this case, we have
$|B'_{ij}|=-B'_{ij}=-B_{ij}+B_{ik}B_{kj}=-|B_{ij}|+|B_{ik}||B_{kj}|$ (note $B_{i,j}>0,B_{ik}B_{kj}>0$ by assumption) as required.

\noindent
Case 2. $B_{k,i}<0$. Then $B_{i,j}<0$ and $B_{j,k}<0$ (so $B_{i,k}>0,B_{j,i}>0,B_{k,j}>0$).
Then, since $\Gamma(B')$ is cyclic, we have $B'_{k,i}>0,B'_{i,j}>0,B'_{j,k}>0$.
By \eqref{eq:matrix-mutation}, we have
$$B'_{ij} =B_{ij}+[B_{ik}]_+ [B_{kj}]_+ - [-B_{ik}]_+ [-B_{kj}]_+=B_{ij}+B_{ik}B_{kj}.$$ 
Thus 
$|B'_{ij}|=B'_{ij} =B_{ij}+B_{ik}B_{kj}=-|B_{ij}|+|B_{ik}B_{kj}|$ as required.

For the second part, suppose $\Gamma(B')$ is acyclic. We assume, without loss of generality, that $sgn(\cc_k)=sgn(B_{k,i})$.
To prove the statement (of the second part), it is enough to show that 
$-|B'_{i,j}|=-|B_{i,j}|+|B_{i,k}||B_{k,j}|$. For convenience, we investigate in cases:

\noindent
Case 1. $B_{k,i}>0$. Then $B_{i,j}>0$ and $B_{j,k}>0$ (so $B_{i,k}<0,B_{j,i}<0,B_{k,j}<0$).
Then, since $\Gamma(B')$ is acyclic, we have $B'_{k,i}<0,B'_{j,k}<0$ but $B'_{i,j}>0$.
By \eqref{eq:matrix-mutation}, we have
$$B'_{ij} =B_{ij}+[B_{ik}]_+ [B_{kj}]_+ - [-B_{ik}]_+ [-B_{kj}]_+=B_{ij}-[-B_{ik}][-B_{kj}]=B_{ij}-B_{ik}B_{kj}.$$ 
Thus, we have $-|B'_{ij}|=-B'_{ij}=-B_{ij}+B_{ik}B_{kj}=-|B_{ij}|+|B_{ik}||B_{kj}|$ as required.

\noindent
Case 2. $B_{k,i}<0$. Then $B_{i,j}<0$ and $B_{j,k}<0$ (so $B_{i,k}>0,B_{j,i}>0,B_{k,j}>0$).
Then, since $\Gamma(B')$ is acyclic, we have $B'_{k,i}>0,B'_{j,k}>0$ but $B'_{i,j}<0$.
By \eqref{eq:matrix-mutation},
$$B'_{ij} =B_{ij}+[B_{ik}]_+ [B_{kj}]_+ - [-B_{ik}]_+ [-B_{kj}]_+=B_{ij}+B_{ik}B_{kj}.$$ Thus 
$-|B'_{ij}|=B'_{ij} =B_{ij}+B_{ik}B_{kj}=-|B_{ij}|+|B_{ik}B_{kj}|$ as required.

\end{proof}

\begin{lemma} 
\label{lem:+-}

Suppose that $(\cc,B)$ is a seed such that $\Gamma(B)$ is \emph{acyclic} and $u$ be the vector whose coordinate vector with respect to the basis $\cc$ is $(\pm d_2|B_{2,3}|,\pm d_3|B_{3,1}|,\pm d_1|B_{1,2}|)$ such that the coordinates corresponding to the source and sink have the same sign and the remaining coordinate has the opposite sign. Let $(\cc',B')=\mu_{{k}} (\cc,B)$. Then we have the following:

If $k$ is a source or sink in $\Gamma(B)$ (so the diagram $\Gamma(B')$ is also acyclic), then the coordinate vector of $u$ with with respect to the basis $\cc'$ is obtained from the one for $\cc$ by multiplying the $k$-th coordinate by  $-1$.


If $k$ is neither a source nor a sink in $\Gamma(B)$ (so the diagram $\Gamma(B')$ is cyclic), then the coordinate vector of $u$ with with respect to the basis $\cc'$ is $(d_2|B'_{2,3}|,d_3|B'_{3,1}|,d_1|B'_{1,2}|)$ or its negative.

\end{lemma} 

\begin{proof}
Let $i,j$ be the remaining vertices (so $\{i,j,k\}=\{1,2,3\}$). 
For the first part, suppose that $k$ is a source or sink, so $sgn(B_{k,i})=sgn(B_{k,j})$. 
Let us denote the $i$-th,$j$-th and $k$-th coordinates of $u$ by $a_i,a_j,a_k$ respectively, so 
$|a_i|=d_k|B_{k,j}|=d_j|B_{j,k}|$, $|a_j|=d_k|B_{k,i}|=d_i|B_{i,k}|$. Then, in particular, 

(*) $|a_i||B_{k,i}|=|a_j||B_{k,j}|$. 

Let us assume, without loss of generality, that $i$ is neither a source nor a sink. Then, by the condition on the signs, the numbers $a_i$ and $a_j$ have opposite signs, so (*) implies that 

(**) $a_i|B_{k,i}|=-a_j|B_{k,j}|$.

Let us assume first that $sgn(\cc_k)=sgn(B_{k,i})=sgn(B_{k,j})$. Then
$\cc'_i=\cc_i+|B_{k,i}|\cc_k$, $\cc'_j=\cc_j+|B_{k,j}|\cc_k$, $\cc'_k=-\cc_k$. 
Then 
the $k$-th coordinate of $u$ with respect to $\cc'$ will be $-a_k$ and the other coordinates are the same because
$u=a_k\cc_k+a_i\cc_i+a_j\cc_j=-a_k(-\cc_k)+a_i(\cc_i+|B_{k,i}|\cc_k)+a_j(\cc_j+|B_{k,j}|\cc_k)$ because $a_i|B_{k,i}|=-a_j|B_{k,j}|$ by (**).

Let us now assume that $sgn(\cc_k)=-sgn(B_{k,i})=-sgn(B_{k,j})$. Then $\cc'_k=-\cc_k$ and $\cc'_i=\cc_i$, $\cc'_j=\cc_j$. Then the $k$-th coordinate of $u$ with respect to $\cc'$ will be $-a_k$ and the other coordinates are the same.

For the second part, suppose that $k$ is neither a source nor a sink, so $sgn(B_{k,i})=-sgn(B_{k,j})$.
We may assume, without loss of generality, that $sgn(\cc_k)=sgn(B_{k,i})$. Then
$\cc'_i=\cc_i+|B_{k,i}|\cc_k$, $\cc'_j=\cc_j$, $\cc'_k=-\cc_k$. 
Let us denote the $i$-th,$j$-th and $k$-th coordinates of $u$ (with respect to $\cc$) by $a_i,a_j,a_k$ respectively, so 
$|a_i|=d_k|B_{k,j}|=d_j|B_{j,k}|$, $|a_j|=d_k|B_{k,i}|=d_i|B_{i,k}|$, $|a_k|=d_i|B_{i,j}|=d_j|B_{j,i}|$ such that 

$sgn(a_i)=sgn(a_j)=-sgn(a_k)$ (***).

\noindent
Then the $k$-th coordinate of $u$ with respect to $\cc'$ will be $a'_k=-a_k+a_i|B_{k,i}|$ and the other coordinates $a'_i,a'_j$ are the same because
$u=a_k\cc_k+a_i\cc_i+a_j\cc_j=(-a_k+a_i|B_{k,i}|)(-\cc_k)+a_i(\cc_i+|B_{k,i}|\cc_k)+a_j(\cc_j)$. 
Note that $sgn(-a_k+a_i|B_{k,i}|)=sgn(a_i)=sgn(a_j)$ by (***).
Thus we may assume, without loss of generality that, $sgn(a_i)=sgn(a_j)=+1=-sgn(a_k)$, so 
$a_i=d_k|B_{k,j}|=d_j|B_{j,k}|$, $a_j=d_k|B_{k,i}|=d_i|B_{i,k}|$, $a_k=-d_i|B_{i,j}|=-d_j|B_{j,i}|$
and show that 
$a'_k=-a_k+a_i|B_{k,i}|=d_i|B_{i,j}|+ d_k|B_{k,j}||B_{k,i}| = d_i|B_{i,j}|+ d_i|B_{k,j}||B_{i,k}| = d_i(|B_{i,j}|+ |B_{k,j}||B_{i,k}|)=
d_i|B'_{i,j}|$, i.e. show that
$|B'_{i,j}|=|B_{i,j}|+ |B_{k,j}||B_{i,k}|$. 
This will complete the proof. For convenience, we investigate in cases.


\noindent
Case 1. $B_{k,i}>0$. Then $B_{j,i}>0$ and $B_{j,k}>0$ (so $B_{i,k}<0,B_{i,j}<0,B_{k,j}<0$).
Note then that, since $\Gamma(B')$ is cyclic, we have $B'_{k,i}<0,B'_{i,j}<0$, $B'_{j,k}<0$.
By \eqref{eq:matrix-mutation},
$$B'_{ij} =B_{ij}+[B_{ik}]_+ [B_{kj}]_+ - [-B_{ik}]_+ [-B_{kj}]_+=B_{ij}-[-B_{ik}][-B_{kj}]=B_{ij}-B_{ik}B_{kj}.$$
Thus $|B'_{ij}|=-B'_{ij}=-B_{ij}+B_{ik}B_{kj}=|B_{ij}|+|B_{ik}||B_{kj}|$ as required.

\noindent
Case 2. $B_{k,i}<0$. Then $B_{j,i}<0$ and $B_{j,k}<0$ (so $B_{i,k}>0,B_{i,j}>0,B_{k,j}>0$).
Then, since $\Gamma(B')$ is cyclic, we have $B'_{k,i}>0,B'_{i,j}>0$, $B'_{j,k}>0$.
By \eqref{eq:matrix-mutation},
$$B'_{ij} =B_{ij}+[B_{ik}]_+ [B_{kj}]_+ - [-B_{ik}]_+ [-B_{kj}]_+=B_{ij}-[B_{ik}][B_{kj}]=B_{ij}+B_{ik}B_{kj},$$ so
$|B'_{ij}|=B'_{ij}=B_{ij}+B_{ik}B_{kj}=|B_{ij}|+|B_{ik}||B_{kj}|$ as required.

\end{proof}



We can now prove our results.


\noindent
Proof of Theorem~\ref{th:cyclic}: Suppose that $\mathbf{i} = (i_1, \ldots, i_l)$ is a maximal green sequence for $B$. Let $(\cc',B')=\mu_{i_{l}} \circ \cdots \circ \mu_{i_1}(\cc,B)$.
Then $\cc'_{j}<0$ for all $j=1,...,n$. Let $u_0$ be the vector (whose coordinate vector with respect to the initial basis $\cc$ is) $(d_2|B_{2,3}|,d_3|B_{3,1}|,d_1|B_{1,2}|)$. Let $(a_1,a_2,a_3)$ be the coordinates of $u_0$ with respect to $\cc'$. By Lemma~\ref{lem:++}, the coordinates $a_1,a_2,a_3>0$. This implies that, since $u_0=a_1\cc'_{1}+a_2\cc'_{2}+a_3\cc'_{3}$, the coordinates of $u_0$ with respect to $\cc$ are non-positive; which is a contradiction. 

\noindent
Proof of Theorem~\ref{th:acyclic}: Suppose that $\Gamma(B_j)$ is cyclic for all $0\leq j\leq l$. Let $(\cc',B')=\mu_{i_{l}} \circ \cdots \circ \mu_{i_1}(\cc,B)$. Then $\cc'_{j}<0$ for all $j=1,...,n$. Let $u_0$ be the vector (whose coordinate vector with respect to the initial basis $\cc$ is) $(d_2|B_{2,3}|,d_3|B_{3,1}|,d_1|B_{1,2}|)$. Let $(a_1,a_2,a_3)$ be the coordinates of $u_0$ with respect to $\cc'$. By the first part of Lemma~\ref{lem:++}, the coordinates $a_1,a_2,a_3>0$. This implies that, since $u_0=a_1\cc'_{1}+a_2\cc'_{2}+a_3\cc'_{3}$, the coordinates of $u_0$ with respect to $\cc$ are non-positive; which is a contradiction.

\noindent
Proof of Theorem~\ref{th:cylic+acyclic}: By Lemmas~\ref{lem:++} and \ref{lem:+-}, the coordinate vector of $u$ with respect to $\cc'$ is $u'$ or $-u'$.
Then the conclusions follow.

\end{document}